\documentclass[11pt,english]{amsart}
\usepackage{amsmath, amsthm, latexsym, amssymb,url,dsfont}
\usepackage{amsfonts, bbold}
\usepackage[all]{xy}
\usepackage{epsfig}
\usepackage{enumerate}
\usepackage{xcolor}
\usepackage{graphicx}
\usepackage{yhmath}
\usepackage{mathdots}
\usepackage{mathtools}
\usepackage{pifont}

\usepackage{tikz}
\usetikzlibrary{matrix,arrows,decorations.pathmorphing}
\usepackage{tikz-cd}

\newcommand{\shrinkmargins}[1]{
  \addtolength{\textheight}{#1\topmargin}
  \addtolength{\textheight}{#1\topmargin}
  \addtolength{\textwidth}{#1\oddsidemargin}
  \addtolength{\textwidth}{#1\evensidemargin}
  \addtolength{\topmargin}{-#1\topmargin}
  \addtolength{\oddsidemargin}{-#1\oddsidemargin}
 \addtolength{\evensidemargin}{-#1\evensidemargin}
  }

\shrinkmargins{.6}

\theoremstyle{plain}

\newtheorem{theorem}{Theorem}[section]
\newtheorem{corollary}[theorem]{Corollary}
\newtheorem{lemma}[theorem]{Lemma}

\newtheorem*{teo}{Theorem}

\newtheorem{definition}[theorem]{Definition}

\theoremstyle{remark}
\newtheorem{remark}[theorem]{Remark}

\theoremstyle{definition}

\def \Z { \mathbb{Z}}

\def \Gal { \text{Gal}}
\def \ker { \text{Ker}}

\def \Tr { {\rm Tr}}

\newcommand{\oo}{{\mathfrak{o}}}

\newcommand{\QQ}{{\mathbb{Q}}}
\newcommand{\ZZ}{{\mathbb{Z}}}

\DeclareMathOperator{\lcm}{lcm}

\begin{document}

\thispagestyle{empty}
\setcounter{tocdepth}{7}

\title{The trace form over cyclic number fields}
\author{Wilmar Bola\~nos \and  Guillermo Mantilla-Soler}
\date{}

\maketitle

\begin{abstract}
In the mid 80's Conner and Perlis showed that for cyclic number fields of prime degree $p$ the isometry class of integral trace is completely determined by the discriminant. Here we generalize their result to tame cyclic number fields of arbitrary degree. Furthermore, for such fields, we give an explicit description of a Gram matrix of the integral trace in terms of the discriminant of the field.
\end{abstract}

\section{Introduction}

An interesting arithmetic invariant of a number field $K$ is its integral trace form, i.e., the integral quadratic form obtained by restricting the bilinear trace pairing \[(x,y) \mapsto {\rm Tr}_{K}(x\cdot y)\] to the maximal order $\oo_{K}$. One of several reasons why the integral trace is of importance in number theory (see \cite{By1, B3, B4, EV1, M5}) is that it is a refinement of the discriminant $\mathfrak{d}_{K}$.  Moreover, by a result of Tausky (see \cite{Ta}) the integral trace is also a refinement of the signature. It follows that two necessary conditions for two number fields $K$ and $L$ to have isometric integral traces is that they have equal degrees and equal discriminants. A result of Conner and Perlis form the early 80's  states that if the fields in question are Galois and of prime degree then such  conditions are also sufficient:

\begin{theorem}\cite[\S IV]{Cp}
Let $p$ be an odd prime and let $K, L$ be two $\Z/p\Z$-number\footnote{Recall that for a finite group $G$ a number field $K$ is called a {\it $G$-number field} if the Galois closure of $K/\QQ$ has Galois group isomorphic to $G$.}fields.  Then \[\left< \oo_K , {\rm Tr}_{K /\QQ} \right> \cong \left< \oo_L , {\rm Tr}_{L /\QQ} \right> \ \mbox{if and only if} \ \mathfrak{d}_{K}=\mathfrak{d}_{L}. \]
\end{theorem}

The objective of this paper is to generalize the above result to cyclic extensions of arbitrary degree. At the moment we can do so under  the additional hypothesis that the fields are {\it tame number fields} i.e., that there is no rational prime that ramifies wildly in either field. Our main result is the following:

\begin{teo}[cf. Theorem \ref{odd} and Theorem \ref{even}]
Let $n$ be a positive integer and let $K, L$ be two tame $\Z/n\Z$-number fields. Then \[\left< \oo_K , {\rm Tr}_{K /\QQ} \right> \cong \left< \oo_L , {\rm Tr}_{L /\QQ} \right> \ \mbox{if and only if} \ \mathfrak{d}_{K}=\mathfrak{d}_{L}. \]
\end{teo}

\begin{remark}
Given two number fields $K, L$, we say they have isometry integral trace, $\displaystyle \left< \oo_K , {\rm Tr}_{K /\QQ} \right> \cong \left< \oo_L , {\rm Tr}_{L /\QQ} \right>,$  if and only if there exist a $\ZZ$-linear isomorphism $\varphi: \oo_K \to \oo_L$ such that
\[{\rm Tr}_{K / \QQ}(x \cdot y) = {\rm Tr}_{L / \QQ}(\varphi(x) \cdot \varphi(y)) \mbox{ for every }\ x,y \in \oo_K. \]
\end{remark}

\begin{remark}
We should note that in the above situation our result implies that the signature of $K$ is determined by its discriminant and its degree; this is not surprising if the degree $n$ is odd since in such a case $K$ is totally real. However, for even $n$ this is saying something not at all obvious not even from the point of view of the genus of the integral trace.
\end{remark}

\subsubsection{A duality between $\Z/n\Z$ and $S_{n}$ number fields}

Let $K$ be a totally real degree $n$ number field, and let $G(K)$ be the Galois group of the Galois closure of $K$ over $\QQ$. If we wanted to try to define a sort of notion of complexity for the field $K$ in terms of the group $G(K)$ we could say that such complexity is very high if $G(K)$ is as big as it can be; i.e., if $G(K) \cong S_{n}.$ At the other side of the spectrum, we could argue that the such complexity is very low if $G(K)$ is as uncomplicated as it can be. For instance, it should have the smallest possible order, $n$, and among those it should have not many automorphisms; for example degree 4 extensions with cyclic Galois group should be ``easier'' than those with Galois group the Klein group. The group $G(K) \cong \Z/n\Z$ meets such requirements. The results presented in this paper are about the behavior of the trace in the low complexity case; in this case, under some ramification assumptions, the trace as an invariant is just the same as the discriminant. In contrast, for the high complexity case (see \cite{M4}) the trace, under some ramification hypotheses as well, is a complete invariant. In other words the strength of the invariant  $\left< \oo_K , {\rm Tr}_{K /\QQ} \right>$ with respect to $\mathfrak{d}(K)$ presents a duality that seems to be determined, at least in the extreme cases, from the complexity of the group $G(K)$. In the recent preprint \cite{PipRob} the authors show that the shape is a complete invariant for $V_4$-quartic fields. As we explained above the complexity of $V_4$-quartic extensions should be greater than that of $\Z/4\Z$-quartics, however intuition says that perhaps it should not be at the same level of $S_4$-quartics. It would be interesting to see if the informal notion of complexity described above really exists or if it is only a fact about $S_n$ and $\Z/n\Z$-extensions.

\subsubsection{Structure of the paper}

In \S \ref{Notation} we set up the notation, and facts, that we will use later in our proofs regarding Hermitian forms over group rings. Then we start with the proofs of our results. The overall strategy is the following: We know by the Hilbert-Speiser theorem that the fields we study have a normal integral basis (NIB). Using Hermitian forms on abelian groups, and the Kronecker-Weber theorem, we  construct a specific NIB and we show that the Gram matrix of the trace, with respect such a basis, depends solely on the discriminant and the degree of the field. This strategy is executed in several stages; in \S \ref{GenralPrimePower} we deal with number fields of prime power degree, there also dealing with different levels:

\begin{itemize}

\item[(a)] First we deal with number fields of prime power discriminant, and odd degree.

\item[(b)] Then we deal with general discriminants, but still odd degree.

\item[(c)] Then we deal with the case of degree a power of $2$.

\end{itemize}

Finally in \S \ref{General}, using that the number field has a cyclic Galois group,  we do a gluing construction to pass from prime power degree to general degree. Here too we must make the distinction between odd and even degrees.

\section{Hermitian forms over group rings}\label{Notation}
Let $G$ be a finite abelian group, and let $\ZZ[G]$ be the group ring of $G$ over $\ZZ$.
Let $X \to \overline{X}$ be the usual involutary ring automorphism of $\ZZ[G]$ such that for every $g \in G,$ $g \to \overline{g} = g^{-1}$. The projection map, which is a morphism of $\ZZ[G]$-modules, and the augmentation map, which is a ring homomorphism, are given by
\begin{eqnarray*}
{\rm Pr}: \ZZ[G] \longrightarrow \ZZ \quad& \qquad \qquad & \qquad \varepsilon: \ZZ[G] \longrightarrow \ZZ\\
{\rm Pr} \left(\sum_{g \in G} d_g g\right) = d_{e}, & \qquad \qquad & \varepsilon \left(\sum_{g \in G} d_g g\right) = \sum_{g \in G} d_{g},
\end{eqnarray*}
where $e \in G$ denotes the identity of $G$. A {\it Hermitian form}
on a left $\ZZ[G]$-module $M$ is a  $\ZZ$-bilinear map
\[ H: M \times M \to \ZZ[G] \]
such that for all $X \in \ZZ[G]$ and $m_1, m_2 \in M$:
\[H(X m_1, m_2) = XH(m_1,m_2),\]
\[ H(m_1, m_2) = \overline{H(m_2, m_1)}.\]
Notice that, since $G$ is abelian, the two conditions above imply that $H(m_1, \overline{X}m_2) =H(X m_1, m_2).$ For example, if
\[ \beta: M \times M \to \ZZ \]
is a $\ZZ$-bilinear and symmetric form such that
\[\beta(X m_1, m_2) = \beta(m_1,\overline{X}m_2)\]
then $\beta$ induces an Hermitian form $H$ given by
\[H(m_1,m_2) = \sum_{g \in G} \beta(g^{-1}m_1, m_2)g.\]

\begin{definition} A {\em symmetric circulant} $\beta$ on $\ZZ[G]$ is a $\ZZ$-bilinear and symmetric form on $\ZZ[G]$ with values in $\ZZ$
such that \[ \beta(gX, gY) = \beta(X,Y) \]
for every $g \in G$ and $X,Y \in \ZZ[G]$.
\end{definition}

If $\beta $ is a symmetric circulant and $H$ the Hermitian form induced by $\beta$, then
\[H(X,Y)= H(e,e)X \overline{Y},\]
for every $X,Y \in \ZZ[G]$. If we denote by
\[ s:= H(e,e) = \sum_{g\in G} \beta(g^{-1}, e) g = \sum_{g \in G} \beta(e, g)g \]
then $s = \bar{s}$ and ${\rm Pr}(sX\overline{Y}) = \beta(X,Y).$ Thus, we have a $1$-to-$1$ correspondence between symmetric circulants and the elements $s$ of the group ring such that
$s = \overline{s} \in \ZZ[G]$. We call $s$ the circulant associated to $H$ or $\beta$  in $\ZZ[G]$. If $\beta$ and $\beta_1$ are symmetric circulants on $\ZZ[G]$ we may ask if there is a $\ZZ[G]$-module automorphism
\[ L: \ZZ[G] \simeq \ZZ[G] \]
such that
\[\beta_{1}(L(X), L(Y)) = \beta(X,Y).\]
Surely $L(X) = XL(e) = XV$ for some unit $V \in \ZZ[G]^{*}$. Thus, for all $X,Y \in \ZZ[G]$
\begin{eqnarray*}
  {\rm Pr}(s_1V \overline{V} X \overline{Y}) &=& \beta_1(L(X), L(Y)) \\
   &=& \beta(X,Y) \\
   &=& {\rm Pr}(sX\overline{Y}).
\end{eqnarray*}

Hence, \[s=s_1 V \overline{V} = V s_1 \overline{V} .\]

In this case we say { \it $s_1$ is congruent to $s$}. This is an equivalence relation on the set of elements $s$ of the group ring such that $s = \overline{s} \in \ZZ[G]$. Thus, the classification of circulants, up to  isometry, is equivalent to the classification of such elements $s$ up to congruence.

\subsection{Induced circulants}
In this subsection we collect some of the basic results about circulants that we will need later in the paper. We do not give proofs of most of the results. For the interested reader proofs can be found in \cite[\S IV.2-3]{Cp}.\\

Let $H \subset G$ be a subgroup, $|H| = h$ and $\chi$ the canonical quotient homomorphism
\[ \chi: G \to G/H.\]

\noindent Then $\chi$ induces a ring homomorphism between group rings

\[ \chi: \ZZ[G] \twoheadrightarrow \ZZ[G/H].\]

We set $\Sigma_H = \sum\limits_{h \in H} h \in \ZZ[G]$, and note that for $X \in \ZZ[H]$, $X \Sigma_H = \varepsilon(X) \Sigma_H$.

\begin{lemma}\label{invariant}
The principal ideal $\left< \Sigma_H \right> \subset \ZZ[G]$ is the ideal of elements fixed under $H$; that is, $X \in \left< \Sigma_H \right>$ if and only if $hX = X$ for all $h \in H$.
\end{lemma}

\begin{lemma}\label{lemm0}
The kernel of
\[\chi: \ZZ[G] \to \ZZ[G/H] \]

is the annihilator ideal in $\ZZ[G]$ of $\Sigma_H$.
\end{lemma}

Lemma \ref{lemm0} means the $\ZZ[G]$-module structure of $\left< \Sigma_H \right> \subset \ZZ[G]$
naturally induces a $\ZZ[G/H]$-module structure on this principal ideal. Furthermore,

\[ X\Sigma_H \to \chi(X) \]
is a $\ZZ[G/H]$-module isomorphism of $\left< \Sigma_H \right>$ with $\ZZ[G/H]$.

Now, suppose that $\beta(X,Y)$ is a symmetric circulant on $\ZZ[G]$ and $s = \overline{s} \in \ZZ[G]$ is the associated circulant for which
\[\beta(X,Y) = {\rm Pr}(s X \overline{Y}).\]

Additionally, note that $\chi(s) = \overline{\chi(s)}$. Thus, the image $\chi(s) \in \ZZ[G/H]$ induces a symmetric circulant on $\ZZ[G/H]$. We seek an interpretation of this induced circulant.

\begin{lemma}\label{2.4}
For $X,Y \in \ZZ[G]$ we have

\[\beta\left( X \Sigma_H, Y \right) = {\rm Pr}(\chi(s) \chi(X) \chi(\overline{Y})).\]

\end{lemma}

Furthermore, note that

\begin{eqnarray*}
  \beta(X\Sigma_H, Y \Sigma_H) &=& \beta(X \Sigma_H^2, Y) \\
   &=& \beta(h X \Sigma_H, Y) \\
   &=& h \beta(X \Sigma_H , Y).
\end{eqnarray*}
Also $\displaystyle \chi: \ZZ[G] \to \ZZ[G/H] $ sends the congruence class of $s = \overline{s} \in \ZZ[G]$ to the congruence class of $\chi(s) \in \ZZ[G/H].$

\subsection{Product of circulants}

Let $G_1$ and $G_2$ be finite abelian groups. Using the inclusions $G_1 \subset G_1 \times G_2$; $ G_{1} \to G_1 \times \{ e \}$
and  $G_2 \subset G_1 \times G_2$; $ G_{2} \to \{e \} \times G_2$  we obtain maps $\ZZ[G_{i}] \subset \ZZ[G_{1} \times G_{2}]$. This yields a bilinear form
\[ \ZZ[G_1] \times \ZZ[G_2] \to \ZZ[G_1 \times G_2] \]
given by
\[ (X_1 , X_2) \to X_1 X_2 \in \ZZ[G_1 \times G_2]\]
and further,  an isomorphism
\[ \ZZ[G_1] \otimes_{\ZZ} \ZZ[G_2] \simeq \ZZ[G_1 \times G_2] \]
which sends $X_1 \otimes X_2$ to $X_1X_2$. Therefore, if $X_1 \in \ZZ[G_1] \subset \ZZ[G]$ and $X_2 \in \ZZ[G_2] \subset \ZZ[G]$ then,
\[{\rm Pr}_G(X_1 X_2) = {\rm Pr}_{G_1}(X_1) {\rm Pr}_{G_2}(X_2) \in \ZZ.\]

Hence, the following:

\begin{lemma}
If $s_1 = \overline{s_1} \in \ZZ[G_1]$ is a circulant associated to $\beta_1$ and $s_2 = \overline{s_2} \in \ZZ[G_2]$ is associated to $\beta_2$ then
$s = s_1 s_2 \in \ZZ[G]$ is canonically associated to the product circulant $\beta_1 \otimes \beta_2$.
\end{lemma}

\section{Prime power degree}\label{GenralPrimePower}

Let $q$ be a prime and $r$ be a positive integer. In this section we consider tame cyclic number fields of degree  $q^{r}.$ The main goal of this section is to present a canonical Gram Matrix of the quadratic module $\langle \oo_K , \Tr_{K/\QQ} \rangle$ that depends only on the degree and the discriminant of $K$.\\

We state the following well known result since we will use it often.

\begin{lemma}{\label{conductor}}
Let $K/\QQ$ be an abelian extension. Suppose that $K$ is tame. Then, the conductor of $K$ is $f={\rm rad}(\mathfrak{d}(K))$.
\end{lemma}

\begin{proof}
This follows from the fact that the conductor is the product over ramified primes of the local conductors. See also  \cite[Proposition 8.1]{Na}.
\end{proof}

\subsection{One Prime Ramifying}
Throughout $K$ denotes a tame cyclic number field of degree $q^r$, and $p \neq q$ the only prime ramifying in $K$.\\

\noindent

Thanks to Lemma \ref{conductor} we know that $ K \subset \QQ(\eta_{p})$, where $\eta_p$ is a primitive $p$-th root of unity. Furthermore, since $\Gal(\QQ(\eta_p) / \QQ) \simeq (\ZZ/p\ZZ)^*$ is cyclic, $K$ is the only subfield of $\QQ(\eta_p)$ of degree $q^r$. Also $q^r | (p-1).$
The ring of integers of $\QQ(\eta_p)$ is $\ZZ[\eta_p]$ and $\{\eta_p , \eta_p^2, \dots, \eta_p^{p-1} \}$ is a normal integral basis of $\ZZ[\eta_p]$. We endow $\ZZ[\eta_p]$ with a structure of a  $\ZZ[(\ZZ/p\ZZ)^*]$-module in the following way:
\[ g\cdot \eta_{p} = \eta_{p}^{g} \]
for every $g \in (\ZZ/p\ZZ)^*$ and extend by linearity. Thus, we have an isomorphism of $\ZZ[(\ZZ/p\ZZ)^*]$-modules of rank 1.

\begin{eqnarray}
\varphi:\ZZ[(\ZZ/p\ZZ)^*] &  \to &  \ZZ[\eta_p] \\
X = \sum a_ig_i  & \to & \sum a_i g_i\cdot \eta_p. \nonumber
\end{eqnarray}
Additionally, we define a symmetric circulant $\beta$ on $\ZZ[(\ZZ/p\ZZ)^*]$ by
\[\beta(X, Y) := \Tr_{\QQ(\eta_p)/ \QQ}(\varphi(X) \varphi(Y))\]

\begin{lemma} Suppose that $t$ is a generator of $(\ZZ/p\ZZ)^*$. Then, the associated circulant of $\beta$ is \[ s = p t^{(p-1)/2} - \Sigma_{(\ZZ/p\ZZ)^*} \]
and $t^{(p-1)/2} \in (\ZZ/p\ZZ)^*$ is independent of the choice of $t$.
\end{lemma}

\begin{proof} The relationship between $s$ and $\beta$ is given by
\[ s = \overline{s} = \sum_{g \in (\ZZ/p\ZZ)^*} \beta(I,g)g.\]
Our goal is to calculate the coefficients $\beta(I,g)$ for every $g$. For this purpose, let $t$ be a generator of $(\ZZ/p\ZZ)^*$, then each $g \in (\ZZ/p\ZZ)^*$ is equal to $t^{j}$ for some $ 1 \leq j \leq p-1 $.
Now, suppose that $\varphi(t) = t(\eta_p) = \eta_p^{r}$ for some $ 1 < r \leq p-1 $, then $\varphi(t^j) = \eta_p^{r^j}$ and

\begin{eqnarray*}
\beta(I,t^j) & = & \Tr_{\QQ(\eta_p)/ \QQ} (\eta_p \cdot \eta_p^{r^j}) \\
& = & \Tr_{\QQ(\eta_p)/ \QQ}(\eta_p^{r^j + 1}) \\
& = & \left\{ \begin{array}{lc} p-1 & \mbox{ if } \; r^j + 1 \equiv 0 \mod p \\ -1 & \mbox{ otherwise.}  \end{array} \right.
\end{eqnarray*}

\noindent But, $r^j + 1 \equiv 0 \mod p$ only when $j = \frac{p-1}{2}$ no matter the choice of the generator $t$. Hence, we have

\begin{eqnarray*}
s & = & \sum_{g \in (\ZZ/p\ZZ)^*} \beta(I,g) g \\
& = & \sum_{j = 0}^{p-1} \beta(I,t^j) t^j \\
& = & p t^{(p-1)/2} - \Sigma_{(\ZZ/p\ZZ)^*}.
\end{eqnarray*}
\end{proof}

\begin{lemma}{\label{oneprime}}
Let $K$ be a cyclic number field of degree $q^r$ with discriminant dvivisible by only one prime $p \neq q$, and let $h = \frac{p-1}{q^r}.$ Then, there is a normal integral basis of $\oo_K$ such that the Gram matrix of the trace in such basis is equal to 
\[\begin{pmatrix} p - h & -h & \dots & -h \\ -h & p - h & \dots & - h \\ \vdots & \vdots & \ddots &  \vdots \\  - h & - h & \dots & p - h  \end{pmatrix}\] if $K$  is totally real, or equal to

\[ \left( \begin{array}{ccc|ccc}  - h &  \dots & -h & p -h &  \dots & -h  \\
\vdots &  \ddots & \vdots &  \vdots  & \ddots & \vdots \\
- h &  \dots& -h  & -h  & \dots & p-h \\ \hline
p -h &  \dots & -h &  - h  & \dots & -h \\
 \vdots  & \ddots & \vdots &   \vdots   & \ddots & \vdots  \\
 - h  & \dots& p-h  & -h & \dots & -h \end{array} \right)\]
if $K$ is totally complex.
\end{lemma}

\begin{proof} Since $p$ is the only prime ramifying in $K$, Lemma \ref{conductor} says that $K \subset \QQ(\eta_p)$. Thus, for every $x,y \in \oo_K$, by transitivity of the trace
\begin{eqnarray*}
\Tr_{\QQ(\eta_p) / \QQ}(xy) & =  &\Tr_{K/ \QQ}\left( \Tr_{\QQ(\eta_p)/ K}(xy) \right) \\
& = &\Tr_{K/ \QQ}(h xy)\\
& = & h \cdot \Tr_{K/ \QQ}(xy).
\end{eqnarray*}
If $t$ is a generator of $G = \Gal(\QQ(\eta_p)/ \QQ)$ then  $H = \left< t^{q^r} \right>$ is the only subgroup of $(\ZZ/p\ZZ)^*$ of order $h:=\frac{p-1}{q^r}$, $K$ is the fixed field of $H$ and $K$ is totally real if and only if $|H| = \frac{p-1}{q^r}$ is even. The last part  part follows since otherwise $K$ would be totally complex. In addition, if $h$ is even, then $t^{(p-1)/2} = (t^{q^r})^{h/2} \in H$, ${\rm otherwise}$ $t^{(p-1)/2} \notin H$.\\

On the other hand, under the isomorphism $\varphi$ defined in $(1)$ we have $\varphi(\left< \Sigma_H\right>) = \oo_K$ and if we set $\mathbb{e}_{1} = \Tr_{\QQ(\eta_p) / K}(\eta_p)$ then the action of $\Gal(K/\QQ)$ over $\mathbb{e}_1$ generates a normal integral basis $\mathbb{e}:=\{\mathbb{e}_1, \mathbb{e}_2, \dots, \mathbb{e}_{q^r} \}$ of $\oo_K$.\\

Denote by $\chi$ the epimorphism from $\Gal(\QQ(\eta_p) / \QQ)$ to $\Gal(K / \QQ)$

\[ \chi: G \to G/H \]
\[ \sigma \to \sigma \restriction_{K} \]

\noindent and extend this to a ring homomorphism

\[\chi: \ZZ[G] \to \ZZ[G/H] \simeq \ZZ[\ZZ/q^r\ZZ].\]

\noindent Under this homomorphism we obtain

\begin{eqnarray*}
\chi\left(\Sigma_G \right) & = & h \Sigma_{G/H},\\
\chi(t^{(p-1)/2}) & = & I_{G/H}, \; \mbox{\, if } h \mbox{ is even,} \\
\chi(t^{(p-1)/2}) & = & \overline{C}, \qquad \mbox{ if } h \mbox{ is odd,}
\end{eqnarray*}

\noindent where $I$ is the identity map and $\overline{C}$ is the conjugation map $\overline{C}: K \to K$. Hence,
\begin{eqnarray*}
\chi(s) & = & pI_{G/H} - h \Sigma_{G/H} \qquad \mbox{\;\, if } h \mbox{ is even,}  \\
 \chi(s) & = & p \overline{C} - h \Sigma_{G/H} \qquad \qquad\mbox{ if } h \mbox{ is odd.}
\end{eqnarray*}
From Lemmas \ref{invariant}, \ref{lemm0} and \ref{2.4} the following diagram is commutative,

\[ \xymatrix{
 & \left<\ZZ[G], \beta \right> \ar[rr]^{\varphi} \ar[dl]_{\chi} & & \left<\ZZ[\eta_p], \Tr_{\QQ(\eta_p) / \QQ}( ) \right>  \\
\left<\ZZ[\ZZ/q^r\ZZ], h \ \widetilde{\beta} \right>  \ar[r]^{\simeq} & \left<\left( \Sigma_{H} \right),\beta \right> \ar@{^(->}[u] \ar[rr]^{\simeq} & & \left< \oo_K , h \ \Tr_{K / \QQ}() \right> \ar@{^(->}[u]
}\]

\noindent we conclude that $\chi(s)$ is the circulant associated to $\Tr_{K/\QQ}$ in the basis $\mathbb{e}$. Therefore the Gram matrix of the trace in the basis $\mathbb{e}$ is

\[\begin{pmatrix} p - h & -h & \dots & -h \\ -h & p - h & \dots & - h \\ \vdots & \vdots & \ddots &  \vdots \\  - h & - h & \dots & p - h  \end{pmatrix} \quad \mbox{ if }\, K \, \mbox{is totally real,}\]

and

\[\left( \begin{array}{ccc|ccc}  - h &  \dots & -h & p -h &  \dots & -h  \\
\vdots &  \ddots & \vdots &  \vdots  & \ddots & \vdots \\
- h &  \dots& -h  & -h  & \dots & p-h \\ \hline
p -h &  \dots & -h &  - h  & \dots & -h \\
 \vdots  & \ddots & \vdots &   \vdots   & \ddots & \vdots  \\
 - h  & \dots& p-h  & -h & \dots & -h \end{array} \right) \quad \mbox{ if }\, K \, \mbox{is totally complex.}\]

\end{proof}

\subsection{Several primes ramifying}
Throughout this section $K$ will denote a tame cyclic number field of degree $q^r$. Let $p_1, p_2, \dots, p_{n}$  be the primes ramifying in $K$. We will denote by $e_{p_i}$ the ramification index of $p_{i}$ in $K$. By Lemma \ref{conductor}, $f= {\rm rad}(\mathfrak{d}(K))$ is the conductor of $K$. The following diagram illustrates this situation:

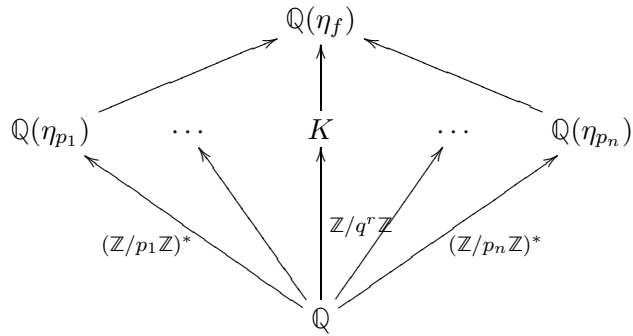
\[\displaystyle \xymatrix{ & & \QQ(\eta_f) & &  \\
\QQ(\eta_{p_1}) \ar[rru]& \cdots & K\ar[u] & \cdots & \QQ(\eta_{p_{n}}) \ar[llu] \\
& &  & & \\
& & \QQ \ar[uu]_{\ZZ/q^r\ZZ} \ar[rruu]_{(\ZZ / p_{n} \ZZ)^{*}} \ar[lluu]^{(\ZZ / p_1 \ZZ)^{*}} \ar[ruu] \ar[luu] &  &}
\]

We denote by $\chi: \Gal( \QQ(\eta_f) / \QQ)  \to  \Gal(K/ \QQ)$ the canonical group homomorphism

\begin{eqnarray*}
 \chi: (\ZZ/ f\ZZ)^* & \to & \ZZ/ q^r\ZZ \\
\sigma & \to & \sigma \restriction_K
\end{eqnarray*}

\begin{lemma}\label{restriction} For every $p_i$ the image of the restriction
$\chi \restriction_{(\ZZ / p_i \ZZ)^{*}} : (\ZZ / p_i \ZZ)^{*} \to \ZZ/q^r\ZZ$ is  $\mathcal{V}_{p_i}(K / \QQ)$, the ramification group of $p_i$ in $K/\QQ$.
\end{lemma}

\begin{proof} Let $P$ be a prime in $\ZZ[\eta_f]$ above of $p_i$ and $P_i:= \ZZ[\eta_{p_i}] \cap P$. Since since $p_i$ is \text{unramified} in $\QQ(\eta_{p_j})$ with $j \neq i$, it is unramified in their compositum. Under the Galois correspondence, this latter field corresponds to the subgroup ${\rm Gal}(\QQ(\eta_{p_{i}})/\QQ) \leq {\rm Gal}(\QQ(\eta_{f})/\QQ)$. Since the inertia subfield of any prime above $p_i$ in $\QQ(\eta_{f})$ is the
maximal subfield in which $p_i$ is unramified, see \cite[Proposition 6.8]{Na}, \[\mathcal{V}_P(\QQ(\eta_f)/ \QQ) \leq {\rm Gal}(\QQ(\eta_{p_{i}})/\QQ).\] In fact, this is an equality since the ramification index of $p_{i}$ in $\QQ(\eta_{p_{i}})$ is $\phi(p_{i}).$ Thus we have isomorphisms
\[\mathcal{V}_P(\QQ(\eta_f)/ \QQ)  \simeq (\ZZ / p_i \ZZ)^{*} \simeq \mathcal{V}_{P_i}(\QQ(\eta_{p_i})).\]

On the other hand, the image of the restriction of the canonical map
\[ \chi: (\ZZ / f\ZZ)^{*} \to \ZZ/{q^r\ZZ} \]
\[ \chi(\sigma) = \sigma \restriction_F \]

to $\mathcal{V}_P(\QQ(\eta_f) / \QQ)$ is $\mathcal{V}_{P_i}(K/ \QQ)$. Finally, the composition
\[\displaystyle \xymatrix{ \mathcal{V}_{P_i}(\QQ(\eta_{p_i})) \simeq   (\ZZ / p_i \ZZ)^* \ar[rr]^{\chi\restriction_{(\ZZ/ p_i \ZZ)^*}} \ar@{^(->}[rd] &  & \ZZ/q^r\ZZ   \\
 & \mathcal{V}_{P}(\QQ(\eta_f)) \subset (\ZZ / f\ZZ)^{*} \ar[ru]_{\chi} & }\]

sends $(\ZZ / p_i \ZZ)^*$ onto $\mathcal{V}_{P_i}(K / \QQ).$
\end{proof}

\begin{corollary} Following the notation above, for every $p_i$ with $i=1,...,n$ \[p_{i} \equiv 1 \pmod{e_{p_i}}.\]
\end{corollary}

\begin{proof}

If $p_i$ ramifies in $K$, then $\mathcal{V}_{p_i}(K / \QQ) \subset \ZZ/{q^r\ZZ}$ is not trivial. Additionally, since \[ \chi:  (\ZZ / p_i \ZZ)^* \to \mathcal{V}_{p_i}(K / \QQ) \]
is onto, then taking cardinalities we conclude that $e_{p_i} \big| (p_i-1)$.

\end{proof}

We proceed now to analyze the general case. Remember that $K$ will denote a cyclic number field of degree $q^r$, $q$ a prime number, and $p_1, p_2, \dots, p_{n}$ the primes ramifying in $K$. Also we suppose that $K$ is tame. \\

By Lemma \ref{restriction} we know that for every $i$, $p_i \equiv 1 \mod e_{p_i}$, and the restriction

\[ \chi \restriction{(\ZZ / p_i \ZZ)^{*}}: (\ZZ / p_i \ZZ)^{*}  \twoheadrightarrow \ZZ/e_{p_i}\ZZ \hookrightarrow \ZZ/q^r\ZZ \]

is onto. Additionally, since for every $i$, $(\ZZ/ p_i \ZZ)^*$ is a cyclic group of order $p_i - 1$ and $e_{p_i} \big| (p_i -1)$, then there exist an unique subgroup $H_i \subset \left( \ZZ / p_i \ZZ \right)^{*}$ for which the quotient is cyclic group of order $e_{p_i}$. This subgroup must be the kernel of the restriction of the canonical homomorphism

\[ \chi \restriction{\left(\ZZ/ p_i \ZZ \right)^*} : \left(\ZZ/ p_i \ZZ \right)^* \to \ZZ/q^r\ZZ.\]

Thus the kernel of

\[ \chi: \left(\ZZ / f \ZZ \right)^* \twoheadrightarrow \ZZ/q^r\ZZ \]

contains the product

\[\tilde{H}:= H_1\times H_2 \times \dots \times H_{n},\]

since $\tilde{H} \subset \ker(\chi)$, then $\chi$ factors through the quotient epimorphism
\[(\ZZ/ f \ZZ)^* \twoheadrightarrow (\ZZ/ f \ZZ)^*/ \tilde{H}  \]
to produce
\[\chi' : (\ZZ/ f \ZZ)^*/ \tilde{H} \twoheadrightarrow \ZZ/q^r\ZZ.\]
That is, the following diagram is commutative

\[ \xymatrix{ &(\ZZ / f\ZZ)^*/ \tilde{H} \ar@{->>}[rd]^{\chi'}  & \\ (\ZZ / f\ZZ)^* \ar@{->>}[rr]^{\chi} \ar[ru]& & \ZZ/q^r\ZZ.}
\]

Furthermore, the kernel of $\chi'$ is $ \ker(\chi)/ \tilde{H}
$, and for each $i$ the restriction of $\chi'$ to $ G_i := (\ZZ/ p_i \ZZ)^*/ H_i$  is an isomorphism
\[\chi' \restriction_{G_i}: G_i \simeq \ZZ/e_{p_i}\ZZ.\]

Letting $F$ be the fixed field of $\tilde{H}$ we have that $K \subset F \subset \QQ(\eta_f)$, moreover if for each $i$ we define $F_{i}: = {\rm Fix}(H_i)$ then the Galois correspondence yields
\[F_{i} \subset \QQ(\eta_{p_i}) \subset \QQ(\eta_f);  \ F=F_{1} \dots F_{n}\]
and $G_i \cong \Gal(F_{i}/ \QQ)$.\\

{\small
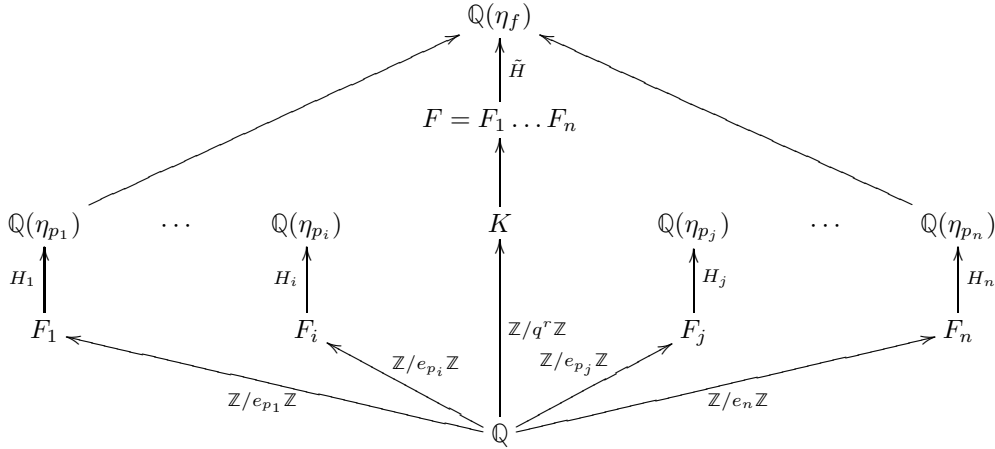
\[ \xymatrix{ &  & &\QQ(\eta_f)& & &  \\
&  & &F = F_1 \dots F_{n} \ar[u]_{\tilde{H}}& &  &  \\
\QQ(\eta_{p_1})\ar[rrruu]& \cdots & \QQ(\eta_{p_i}) &  K \ar[u]&  \QQ(\eta_{p_j}) & \cdots & \QQ(\eta_{p_{n}}) \ar[llluu]  \\
F_1 \ar[u]^{H_1}& &F_i\ar[u]^{H_i} &  &  F_j \ar[u]_{H_j}& &F_{n} \ar[u]_{H_{n}}  \\
&  & &\QQ \ar[rrru]_{\ZZ/e_{n}\ZZ} \ar[lllu]^{\ZZ/e_{p_1}\ZZ}  \ar[ru]^{\ZZ/e_{p_j}\ZZ} \ar[lu]_{\ZZ/e_{p_i}\ZZ} \ar[uu]_{\ZZ/q^r\ZZ} & &  & }\]
}

In particular, each extension $F_{i} / \QQ$ has degree $e_{p_i}$, Galois group $\ZZ/e_{p_i}\ZZ \simeq G_i$, and $p_i$ is the only prime ramifying.

Thus, by Lemma \ref{oneprime} we can find a normal integral basis $\{w_i^{j}\}_j$ of $F_i$  such that
\[\varphi_i : \ZZ[G_i] \simeq \oo_{F_{i}}\]
is given by $\varphi_i(I) = w_i^{1}$ and
\[s_i = p_i Y_i - h_i \Sigma_{G_i}\]
where
\[Y_i : = \left\{ \begin{array}{ll} I & \mbox{\; if \; } F_i \mbox{\; is totally real \;} \\ \overline{C} & \mbox{\; if \; } F_i \mbox{\; is totally complex \; } \end{array} \right.\]
and $h_i := \frac{p_i - 1}{e_{p_i}}$.

Since the discriminants of $\QQ(\eta_{p_i})$ are pairwise coprime, the discriminant of $F_{i} \subset \QQ(\eta_{p_i})$ are pairwise coprime. Hence, since $F$ is the compositum $F_{1} \dots F_{n}$ it follows from \cite[Theorem 4.26]{Na} that
\[\displaystyle F = F_{1} \dots F_{n} = F_{1} \otimes_{\QQ} \dots \otimes_{\QQ} F_{n},\]
and that
\[ \oo_F = \oo_{F_{1}} \dots \oo_{F_{n}} = \oo_{F_{1}} \otimes_{\ZZ} \dots \otimes_{\ZZ} \oo_{F_{n}} .\]

Thus, we can define $\varphi : \ZZ[\tilde{G}] \to \oo_F$ where $ \tilde{G} := G_1 \times \dots \times G_{n},$ such that the following diagram is commutative

\[ \xymatrix{
\ZZ[G_1] \otimes_{\ZZ} \dots \otimes_{\ZZ} \ZZ[G_{n}] \ar[rrr]^{\simeq} \ar[d]_{\bigotimes\limits_{i} \varphi_i } &  & &
\ZZ[\tilde{G}] \ar[d]^{\varphi} \\
\oo_{F_{1}} \otimes_{\ZZ} \dots  \otimes_{\ZZ} \oo_{F_{n}} \ar[rrr]^{\simeq}&  & &
\oo_{F}
}\]

It follows that $ \varphi :\ZZ[\tilde{G}] \to \oo_F$ satisfies
\[\varphi(I) = w_{1} \cdot ... \cdot  w_{n} := w.\]

\begin{remark}\label{LaINBDeF}
Notice that $\tilde{G} = \Gal(F/ \QQ)$ and that the action of $\tilde{G}$ on $w$ generates a normal integral basis for $\oo_F$.
\end{remark}
 Now, since the trace form on $\oo_F$ down to $\ZZ$ is the tensor product of the trace forms on $\oo_{F_i}$, we can extend $\varphi$ to the following quadratic spaces in the following way

\[\xymatrix{
\left<\ZZ[\tilde{G}], \otimes \beta_i \right> \ar[rr]^{\varphi} & & \left<\oo_F, \Tr_{F / \QQ} \right> \\
\bigotimes \left<\ZZ[G_i], \beta_i \right> \ar[u]^{\simeq} \ar[rr]_{\otimes \varphi_i} & & \bigotimes \left< \oo_{F_i}, \Tr_{F_i / \QQ} \right> \ar[u]_{\simeq}
}\]

The resulting circulant, $s = \overline{s} \in \ZZ[\tilde{G}]$, for the symmetric bilinear form $\beta:= \otimes \beta_i$ on $\ZZ[\tilde{G}]$ is $s_1 \cdot ... \cdot s_{n}$.\\

Finally, remember that $\chi': \tilde{G} \twoheadrightarrow \ZZ/q^r\ZZ$ has a kernel $H$, whose fixed field is $K$. Additionally, $\chi'$ induces a ring homomorphism

\[\chi' : \ZZ[\tilde{G}] \twoheadrightarrow \ZZ[\ZZ/q^r\ZZ],\]

with kernel equal to the annihilator ideal in $\ZZ[\tilde{G}]$ of $\Sigma_{H}$. Furthermore,

\[X \Sigma_{\tilde{G}} \to \chi'(X)\]

is a $\ZZ[\ZZ/q^r\ZZ]$-module isomorphism of $\left( \Sigma_{\tilde{G}} \right)$ with $\ZZ[\ZZ/q^r\ZZ]$ and the following diagram is commutative:

\[\xymatrix{
 & \left<\ZZ[\tilde{G}], \beta \right> \ar[rr]^{\varphi} \ar[dl]_{\chi'} & & \left<\oo_F, \Tr_{F / \QQ}( ) \right>  \\
\left<\ZZ[ \ZZ/q^r\ZZ], k \widetilde{\beta} \right>  \ar[r]^{\simeq} & \left<\left( \Sigma_{G'} \right),\beta \right> \ar@{^(->}[u] \ar[rr]^{\simeq} & & \left< \oo_K , k \Tr_{K / \QQ} \right> \ar@{^(->}[u]
}\]

The circulant associated to $\widetilde{\beta}$ is $\chi'(s) = \chi'(s_1)\cdot \chi'(s_2) \cdot ... \cdot \chi'(s_{n})$. Since $\chi' \restriction_{G_i}: G_{i} \simeq  \ZZ/e_{p_i}\ZZ $, by Lemma \ref{oneprime}, we have

 \[\chi'(s_i) = \chi'\left(p_i Y_i - h_i \Sigma_{G_i}\right) = p_i I - h_i \Sigma_{\langle e_{p_i} \rangle} \in \ZZ[\ZZ/q^r\ZZ], \quad \mbox{ if \; } q \mbox{\; is odd}\]

\[ \chi'(s_i) = \chi'\left(p_i Y_i - h_i \Sigma_{G_i}\right) = p_i Y'_i - h_i \Sigma_{\langle e_{p_i}\rangle} \in \ZZ[\ZZ/q^r\ZZ], \quad \mbox{ if \; } q \mbox{\; is even}\]

where $\langle e_{p_i}\rangle $ represents the only subgroup of $\ZZ/q^r\ZZ$ of order $e_{p_i}$, and
\[Y'_i :=  \left\{\begin{array}{ll} I & \mbox{ if } F_i \mbox{ is totally real} \\
\sigma & \mbox{ if } F_i \mbox{ is totally complex } \end{array}\right. \]
where $\sigma$ is the only element of order 2 in $\ZZ/q^r\ZZ$. Thus, if $K$ is totally complex then $\sigma$ is complex conjugation.

\begin{lemma}\label{coef} Let $K$ be a cyclic number field of degree $q^r$, $q$ odd prime. In keeping up with the notation above, let $\{p_{1},...p_{n}\}$ be the set of primes that ramify in $K$, all tame, let $e_{i}$ be the usual ramification index and let $h_i := \frac{p_i - 1}{e_{p_i}}$. Let $s$ and $\chi'$ be as above. Suppose $a_{0},...,a_{r}$ are integers such that \[ \chi'(s) =  \prod\limits_{i = 1}^{n} \left( p_i I - h_i \Sigma_{ \langle e_{p_i} \rangle } \right) = a_0 I + a_1 \Sigma_{\langle q \rangle} + a_2 \Sigma_{\langle q^2 \rangle} + \dots + a_r \Sigma_{\langle q^r \rangle}.\] For each  $1 \leq j \leq r$, let $\mathbb{P}_{j}(K):= \{ p_i : e_{p_i}(K/\QQ) = q^j \}$. If \footnote{as it is standard the product over the empty set is defined to be $1$.} 
 \[m_i :=\prod_{p \in \mathbb{P}_{i}(K)} p \ \mbox{and} \ f_{i} := \frac{m_i - 1}{q^{i}} \]
then,
\[a_0  =  p_1 \cdot p_2 \cdot ... \cdot p_{n}\]
and for $1 \leq i \leq r$
\[a_i  = - f_i \prod\limits_{j > i} m_j.\]

\end{lemma}

\begin{proof} By reindexing, if necessary, we have

\begin{eqnarray*}
& & \prod_{i = 1}^{n} \left( p_i I - h_i \Sigma_{\langle e_{p_i} \rangle} \right) \\
& = & \underbrace{(p_1 I - h_1 \Sigma_{\langle q \rangle}) \dots (p_j I - h_j\Sigma_{\langle q \rangle})}_{e_{p_i} = q} \dots \underbrace{(p_l I - h_l\Sigma_{\langle q^r \rangle}) \dots (p_{n} I - h_{n}\Sigma_{\langle q^r \rangle})}_{e_{p_i} = q^r} \\
& = & \left(m_1 I - f_1 \Sigma_{\langle q \rangle} \right) \left(m_2 I - f_2 \Sigma_{\langle q^2 \rangle} \right)\ldots \left(m_r I - f_r \Sigma_{\langle q^r \rangle} \right).
\end{eqnarray*}
Using that 
\[\Sigma_{\langle q^j \rangle} \Sigma_{\langle q^{j+1} \rangle} = q^j\Sigma_{\langle q^{j+1} \rangle}\]
and that
\[m_j -q^jf_j = 1\]
 we finish the proof by induction on the number of $f_j \neq 0$.
\end{proof}

\begin{corollary}\label{ElCoroDeLaMatriz}

Let $K$ be as in Lemma \ref{coef}. For all $0 \leq i \leq r$ let $A_i$ be the  $q^r \times q^r$ matrix defined by
\[(A_i)_{l,m}:= \left\{ \begin{array}{ll} 1 & \mbox{ if } q^{r-i} | (m-l), \\
 & \\
0 & \ \ {\rm otherwise}.  \end{array} \right.\]

Then, there is an normal integral basis $\mathbb{e}$ of $\oo_K$ such that the Gram matrix of the trace form in such a basis is equal to \[ M = a_0 A_0 + a_1 A_1 + \cdots + a_r A_r.\]

\end{corollary}

\begin{proof}
Let $t$ be a generator of $\Gal(K/\QQ)$, and 
let $w$ be as in Remark \ref{LaINBDeF}. If we let $\mathbb{e}_{1} : = \Tr_{F/K}(w)$, and $\mathbb{e}_{j + 1} := t^j(\mathbb{e}_1)$ for  $j =0,1,2,\dots,q^{r}-1$ then $\mathbb{e}:=\{\mathbb{e}_{1},...,\mathbb{e}_{q^{r}}\}$ is a normal integral basis. Furthermore, for this basis we have
\[\begin{array}{llc} \Tr_{K/ \QQ} (\mathbb{e}_{j +1})  & = ( -1)^{n} & \forall_j \\
 & \\
\Tr_{K / \QQ}(\mathbb{e}_{j + 1} \mathbb{e}_{i + 1}) & = \sum\limits_{k = r - l}^{r} a_k & \mbox{ if }  \ (i-j,q^r)= q^{l} \\
& \\
\Tr_{K / \QQ}(\mathbb{e}_{j + 1} \mathbb{e}_{i + 1}) & = a_r & \mbox{ if }  \ (i-j,q^r)= 1.
\end{array}\] from which the result follows.
\end{proof}

\begin{theorem}\label{main1}
Let $K, K'$ be two tame cyclic number fields of degree $q^r$, where $q$ is an odd prime.  Then, \[\left<\oo_K, \Tr_{K/ \QQ}  \right> \simeq \left<\oo_L, \Tr_{K'/ \QQ} \right> \ \mbox{if and only if} \ \mathfrak{d}(K) = \mathfrak{d}(K').\]

\end{theorem}

\begin{proof}
We show the non trivial implication. Thanks to Corollary \ref{ElCoroDeLaMatriz} we know that $K$ and $K'$ have integral basis such that the Gram matrices of their traces, in their respective basis, are 
$ M = a_0 A_0 + a_1 A_1 + \cdots + a_r A_r $ and
$M' = a'_0 A_0 + a'_1 A_1 + \cdots + a'_r A_r$.
It suffices to show that for all $i$, $a_i = a'_i$. By Lemma \ref{coef} we see that the values $a_{i}$ and $a'_{i}$ are completely determined by the ramification indices of each ramified prime. Hence, it is enough to show that $e_p(K/ \QQ) = e_p(K' / \QQ)$ for all prime $p$. This is indeed the case since for a Galois number field $E$ of degree $n$ and discriminant $\mathfrak{d}(E)$; the exponent of $p$ in $\mathfrak{d}(E)$ is equal to $n(1-\frac{1}{e_p(E/ \QQ)})$ for every prime tame in $E$. Indeed, for all $\mathfrak{P} | p$ in $K$, $\mathfrak{P}^{(e_{p}-1)}$ exactly divides the different of $K/\QQ$, so that $p^{f_{p}g_{p}(e_{p}-1)}$ exactly divides its discriminant

\end{proof}

\begin{lemma}\label{coef2} Let $K$ be a tame cyclic number field of degree $2^r$. Let $\{p_1, p_2, \dots, p_n  \}$ be the set of primes ramifying in $K$, let $h_i:= \frac{p_i - 1}{e_{p_i}}$, where $e_{p_i}$ is the usual ramification index of $p_i$ in $K$ and let $\varepsilon$ be the number of $i$'s such that $e_i||(p_i-1)$. Let $s$ and $\chi'$ be as before. Then, there exist $a_0, \dots,a_r$ integers such that 
\[ \chi'(s) =  \prod\limits_{i = 1}^{n} \chi'(s_i) = a_0 \sigma^{\varepsilon} + a_1 \Sigma_{ \langle 2 \rangle} + a_2 \Sigma_{\langle 2^2 \rangle} + \dots + a_r \Sigma_{\langle 2^r \rangle}.\]
For each $1 \leq j \leq r$, let $\mathbb{P}_j:= \{ p_i : e_{p_i}(K/\QQ) = 2^j \}$. If

\[m_i := \prod_{p \in \mathbb{P}_{i}(K)} p \ \mbox{and} \ f_{i} := \frac{m_i - 1}{2^{i}}\]
then,
\[a_0  =  p_1 \cdot p_2 \cdot ... \cdot p_{n}\]
and for $1 \leq i \leq r$
\[a_i  = - f_i \prod\limits_{j > i} m_j.\]
\end{lemma}

\begin{proof} By reindexing, if necessary, we have
{
\begin{eqnarray*}
& & \prod_{i = 1}^{n} \left( p_i Y'_i - h_i \Sigma_{\langle e_{p_i} \rangle} \right)\\
 & = & \underbrace{(p_1 Y'_1 - h_1 \Sigma_{\langle 2 \rangle}) \dots (p_j Y'_j - h_j\Sigma_{\langle 2 \rangle})}_{e_{p_i} = 2} \dots \underbrace{(p_l Y'_l - h_l\Sigma_{\langle 2^r \rangle}) \dots (p_{n} Y'_{n} - h_{n}\Sigma_{<2^r>})}_{e_{p_i} = 2^r} \\
& = & \left(m_1 \sigma^{\varepsilon_1} - f_1 \Sigma_{\langle 2 \rangle } \right) \left(m_2 \sigma^{\varepsilon_2} - f_2 \Sigma_{\langle 2^2 \rangle} \right)\ldots \left(m_r \sigma^{\varepsilon_r} - f_r \Sigma_{\langle 2^r \rangle} \right).
\end{eqnarray*}
}
Where for every $i$, $\varepsilon_i$ means the number of $i$'s  such that $e_{p_i} || (p_i -1) $ and $e_{p_i} = 2^i$. Using that \[ \varepsilon = \sum_{i=1}^{r} \varepsilon_i \] and that
\[ \Sigma_{\langle 2^j \rangle} \Sigma_{\langle 2^{j+1} \rangle} = 2^j \Sigma_{2^{j+1}} \]
the result follows by induction on the number of $f_j \neq 0$.
\end{proof}

\begin{corollary} Let $K$ as in Lemma \ref{coef2}. For every $0 \leq i \leq r$ let $A_i$ be the $2^r \times 2^r$ matrix defined by
\[ (A_i)_{l,m} :=  \left\{ \begin{array}{ll} 1 & \mbox{ if } 2^{r-i} | (m-l), \\
 & \\
0 & otherwise.  \end{array} \right.\]

Then, there is a normal integral basis $\mathbb{e}$ of $\oo_K$ such that the Gram matrix  of $\left< \oo_K, \Tr_{K / \QQ} \right>$ in such basis is equal to

\[ M = a_0 A_1^{\varepsilon} + a_1 A_1 + \cdots + a_r A_r. \]
\end{corollary}

\begin{proof}
Let $t$ be a generator of $\Gal(K/\QQ)$, and 
let $w$ be as in Remark \ref{LaINBDeF}. If we let $\mathbb{e}_{1} : = \Tr_{F/K}(w)$, and $\mathbb{e}_{j + 1} := t^j(\mathbb{e}_1)$ for  $j =0,1,2,\dots,2^{r}-1$ then $\mathbb{e}:=\{\mathbb{e}_{1},...,\mathbb{e}_{2^{r}}\}$ is a normal integral basis. Furthermore, for this basis we have
\[ \begin{array}{llc} \Tr_{K/ \QQ} (\mathbb{e}_{j +1})  & = ( -1)^{n} & \forall_j \\
 & \\
\Tr_{K / \QQ}(\mathbb{e}_{j + 1} \mathbb{e}_{i + 1}) & = \sum\limits_{k = r - l}^{r} a_k & \mbox{ if }  \ (i-j,2^r)= 2^{l} \\
& \\
\Tr_{K / \QQ}(\mathbb{e}_{j + 1} \mathbb{e}_{i + 1}) & = a_r & \mbox{ if }  \ (i-j,2^r)= 1.
\end{array} \] from which the result follows.
\end{proof}

\begin{theorem} {\label{2tothen}}Let $K,K'$ be two tame cyclic number fields of degree $2^n$. Then,
\[\left< \oo_K , \Tr_{K/\QQ} \right> \simeq  \left< \oo_{K'}, \Tr_{K'/\QQ} \right> \ \mbox{if and only if}\  \mathfrak{d}(K) = \mathfrak{d}(K').\]

\end{theorem}

\begin{proof}
We show the non trivial implication. By the same argument at the end of the proof of Theorem \ref{main1} we see that $e_p(K' / \QQ) = e_p(K/\QQ)$ for every prime $p$. Hence, the respective associated circulants

  \[ s(K / \QQ) = \prod_{i = 1}^{n} \left( p_i Y'_i - h_i \Sigma_{<e_{p_i}>} \right) \]
  and
  \[s(K' / \QQ) = \prod_{i = 1}^{n} \left( p_i Y'_i - h_i \Sigma_{\langle e_{p_i} \rangle} \right) \]
   must be equal, so $\left< \oo_K , \Tr_{K/\QQ}() \right> $ and  $\left< \oo_{K'}, \Tr_{K'/\QQ}() \right> $ are equivalent.

\end{proof}

\section{General degree}\label{General}
In this section, we study the behaviour of $\left< \oo_K , \Tr_{K /\QQ} \right>$  when $\Gal(K / \QQ)$ is cyclic of order $m$, arbitrary, and $K$ tame.
Let $m = q_1^{r_1} \cdot ... \cdot q_l^{r_l}$ be the prime decomposition of $m$. Since $\Gal(K / \QQ)$ is cyclic there exist $G_1, G_2, \dots, G_l \subset \Gal(K / \QQ)$ subgroups, such that $G_i \simeq \ZZ/q_i^{r_i}\ZZ$ and that
\[ \Gal(K / \QQ) = G_1 \cdot G_2 \cdot ... \cdot G_l.\]

By the Galois correspondence, there are fields $K_1, K_2, \dots, K_l \subset K$ such that \[ K = K_1 K_2 \dots K_l,\]
and $\Gal (K_i / \QQ) \simeq \ZZ/q_i^{r_i}\ZZ.$

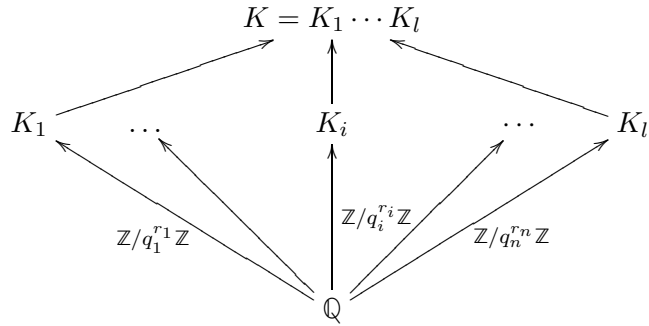
\[ \xymatrix{ & & K = K_1 \cdots K_l & &  \\
K_1 \ar[rru]& \dots & K_i \ar[u] & \cdots & K_l \ar[llu] \\
& &  & & \\
& & \QQ \ar[uu]_{\ZZ/q_i^{r_i}\ZZ} \ar[rruu]_{\ZZ / q_n^{r_n} \ZZ} \ar[lluu]^{\ZZ / q_1^{r_1} \ZZ} \ar[ruu] \ar[luu] &  &}
\]

Furthermore, this decomposition is unique and only depends on $K$. Additionally, every $K_i$ is tame since $K$ is tame by hypothesis.\\

We denote $P_i \subset \{ p_1, p_2, \dots, p_{n} \}$ the set of primes ramifying in $K_i$. 
Since every $K_i$ is a tame cyclic number field of degree $q_i^{r_i}$, then we can proceed as in \S\ref{GenralPrimePower} in order to find number fields $F_{1}^i, \dots , F_{|P_i|}^i$ such that every $F_k^i$ is cyclic, tame, and only one prime is ramifying in $F_k^i$; additionally, we have
\[K_i \subset F_{1}^i \dots F_{|P_i|}^i.\]
We can continue this process with every $K_i$ in order to obtain numbers fields $F_{k}^i$, with $ 1 \leq i \leq l $ and $1 \leq k \leq |P_i|$, each one cyclic, tame, only one prime is ramifying, degree $[F_k^i:\QQ] = e_p(K_i/\QQ)$ and $F_k^i \subset \QQ(\eta_p)$ where $p$ is the only prime ramifying in $F_k^i$.\\

\begin{lemma}{\label{cycliclem}} Let $p_1, p_2, \dots, p_{n}$ be the primes ramifying in $K$, then there exist subfields $F_1 , F_2, \dots F_{n} \subset \QQ(\eta_f)$ such that for every $i$, $F_i \subset \QQ(\eta_{p_i})$, $K \subset F_1 F_2 \dots F_{n}$, and $[F_i : \QQ] = e_{p_i}(K / \QQ)$.
\end{lemma}

\begin{proof}

Based on the construction above, we have number fields $F_{k}^i$, with $ 1 \leq i \leq l $ and $1 \leq k \leq |P_i|$. For every prime $p_j$, we associate the  $F_k^i$, $1 \leq i \leq l, 1 \leq k \leq |P_i|$, according the prime which is ramifying in each one and denote by $F_j$ its composition. Let us check that this construction has the properties that we claimed. Clearly, every $F_j \subset \QQ(\eta_{p_j})$ since by construction the only prime ramifying in $F_j$ is $p_j$, additionally,
\[F_1 F_2 \dots F_{n} = \prod_{i,k}F_k^i \supset K_1 \dots K_l = K.\]
Now, suppose that $p_j$ ramifies only in $K_i$, then $e_{p_j}(K / \QQ) = e_{p_j}(K_i/ \QQ) = [F_i : \QQ]$.\\
On the other hand, if $p_j$ ramifies in $\displaystyle \{K_{j_1}, K_{j_2}, \dots, K_{j_s} \} $ then the respective number field $F_k^{j_l}$ in which $p_j$ 
ramifies has degree $\displaystyle e_{p_j}(K_{j_l} / \QQ)$, since every $K_i$ has degree $q_i^{r_i}$ then the numbers $\displaystyle e_{p_j}(K_{j_l}/ \QQ)$
are mutually  coprimes, then the degree of $F_j$, being the composite of the number field $F_k^{j_l}$ in which $p_j$ is ramifying, is the product of these degrees, but this product is the ramification index $e_{p_j}(K/ \QQ)$.
\end{proof}

From Lemmas \ref{conductor} and \ref{cycliclem}, we have the following diagram

\[ \xymatrix{ & & \QQ(\eta_f) & &\\
& & F:= F_1 \dots F_{n} \ar[u] & & \\
& & K = K_1 \dots K_l \ar[u] & &  \\
K_1 \ar[rru]&  \cdots & K_i \ar[u] & \cdots & K_l \ar[llu] \\
& &  & & \\
& & \QQ \ar[uu]_{\ZZ/q_i^{r_i}\ZZ} \ar[rruu]_{\ZZ / q_l^{r_l} \ZZ} \ar[lluu]^{\ZZ / q_1^{r_1} \ZZ} \ar[ruu] \ar[luu] &  &}
\]

As before, we denote by $\chi$ the canonical projection from $\Gal(\QQ(\eta_f) / \QQ)$ to $\Gal(K / \QQ)$,
\[ \chi : (\ZZ/f\ZZ)^* \to \ZZ/m\ZZ \]
\[\chi(\sigma) = \sigma \restriction_{K}.\]
For every $i$ the ramification group of $p_i$ in $\QQ(\eta_f)$ is
\[\mathcal{V}_{p_i}(\QQ(\eta_f)/ \QQ) = \{id\}\times \dots \times (\ZZ/p_i\ZZ)^* \times \dots \times \{id \},\]
and the restriction of $\chi$ to $\mathcal{V}_{p_i}(\QQ(\eta_f) / \QQ)$ is onto to the ramification group of $K/ \QQ$

\[\chi \restriction{V_{p_i}(\QQ(\eta)/ \QQ)} : V_{p_i}(\QQ(\eta)/ \QQ) \twoheadrightarrow \mathcal{V}_{p_i}(K / \QQ).\]

Now, by Lemma \ref{cycliclem} there exist a subfield $F := F_1 \dots F_{n}$ such that $\QQ \subset  K \subset F \subset \QQ(\eta_f)$, therefore $\chi$ factors trough $\Gal(F/ \QQ)$ into
\[\chi': \Gal(F / \QQ) \to \ZZ/m\ZZ.\]

Our main interest at this point is to understand the behaviour of $\left<\oo_F, \Tr_{F / \QQ} \right>$. Since $F = F_1 F_2 \dots F_{n}$ and for every $i$, $\mathfrak{d}(F_i)$ is a power of $p_i$, then $\mathfrak{d}(F_i)$ is mutually coprime to $\mathfrak{d}(F_j)$ for $j \neq i$, therefore
\[\left< o_F, \Tr_{F / \QQ} \right> = \bigotimes_{i=1}^{n} \left< o_{F_i}, \Tr_{F_i / \QQ} \right>.\]

Thus, it's enough to understand the behaviour of each $\left< o_{F_i}, \Tr_{F_i / \QQ} \right> $ in order to complete our task.

\subsection{Odd degree}

Let us assume that the degree of $K$,  $m =q_1^{r_1} \dots q_l^{r_l}$ is an odd number. For $1 \leq i \leq l$ let $s_i$ be the circulant associated to $F_i$. Since $F_i$ has odd degree, we get
\[ s_i = p_i I - h_i \Sigma_{G^i}  \]
where, $\displaystyle h_i := \frac{p_i - 1}{[F_i:\QQ]} =\frac{{p_i} - 1}{e_{q_i}(K / \QQ)} $  and $\displaystyle G^i := \Gal(F_i / \QQ) \simeq \ZZ/ e_{q_i}(K/\QQ)\ZZ$. \\

Therefore, the circulant $s$ associated to $F$ is

\begin{eqnarray*}
  s &=& s_1 \cdot s_2 \cdot ... \cdot s_{n} \\
   &=& \prod_{i=1}^{n} \left(p_i I - h_i \Sigma_{G^i}.\right)
\end{eqnarray*}

Now, since $\chi': \Gal(F / \QQ) \to \Gal(K/ \QQ)$ is onto, then the circulant associated to $K$ is
\begin{eqnarray*}
 \chi'(s) & = & \prod_{i=1}^{n} \chi'(s_i)\\
& = & \prod_{i=1}^{n} \chi'\left(p_i I - h_i \Sigma_{G^i} \right)\\
& = & \prod_{i=1}^{n} \left( p_i I -h_i \Sigma_{\langle e_{p_i} \rangle} \right).
\end{eqnarray*}
Where $\left< e_{p_i} \right>$ denotes the unique subgroup of $\ZZ / m\ZZ$ of order $e_{p_i}(K / \QQ)$.\\

\begin{theorem}{\label{odd}}

Let $K, K'$ be two tame cyclic number fields of odd degree $m$. Then, \[ \left<\oo_K, \Tr_{K/ \QQ}  \right> \simeq \left<\oo_{K'}, \Tr_{K'/ \QQ} \right> \ 
\mbox{if and only if} \ \mathfrak{d}(K) = \mathfrak{d}(K').\]
\end{theorem}

\begin{proof}
We show the non trivial implication. As usual the hypotheses imply that $e_p(K / \QQ) = e_p(K' / \QQ)$ for every prime $p$. Therefore,  the respective associated circulants

\[s(K) = \prod_{i=1}^{n} \left( p_i I -h_i \Sigma_{\langle e_{p_i} \rangle} \right)  = s(K')\]

are equal and therefore the quadratics modules $\left<\oo_{K}, \Tr_{K / \QQ} \right>$ and $\left<\oo_{K'}, \Tr_{K' / \QQ} \right>$ are isometric.
\end{proof}

We now describe the circulant $s$ associated to a cyclic tame number field $K$ of odd degree $m$. Let $1 = d_1 < d_2 < \dots < d_{\tau(m)} = m$ be the set of positive divisors of $m$, and let
\[ P := \left\{ \left(d_2^{\varepsilon_2}, d_3^{\varepsilon_3}, \dots , d_{\tau(m)}^{\varepsilon_{\tau(m)}} \right) \in \ZZ^{\tau(m)-1} : \varepsilon_i = 0 \mbox{ or } 1 \mbox{ for every } i  \right\}.\]

Additionally, for every $\vec{v} \in P$ we define
\[ \lcm(\vec{v}) := \lcm\left[ d_2^{\varepsilon_2}, d_3^{\varepsilon_3}, \dots , d_{\tau(m)}^{\varepsilon_{\tau(m)}}\right],\]
\[ \gcd(\vec{v}) := \gcd\left(d_2^{\varepsilon_2}, d_3^{\varepsilon_3}, \dots , d_{\tau(m)}^{\varepsilon_{\tau(m)}}\right),\]
and for every $d|m$
\[ P_d := \{ \vec{v} \in P : \lcm(\vec{v}) = d \}. \]
With this notation in mind we state the following Lemma:

\begin{lemma}{\label{coef3}}
Let $K$ be a tame cyclic number field $K$ of odd degree $m$ and let $p_1, p_2, \dots, p_{n}$ be the primes ramifying in $K$. Then, the circulant $s$ associated to $K$ is given by

\begin{eqnarray*}
  s(K) &=& \prod_{i=1}^{n} \left( p_i I -h_i \Sigma_{\langle e_{p_i} \rangle} \right) \\
    &=& \sum_{d | m} a_d \Sigma_{\langle d \rangle}
\end{eqnarray*}
where $\langle d \rangle$ denotes the only subgroup of $\ZZ/ m\ZZ$ of order $d$, and $a_i$ is given by the formula 

\[ \left\{ \begin{array}{rl} a_1 & =  p_1 \cdot p_2 \cdot ... \cdot p_{n} \\   & \vdots \\
a_d & = \sum\limits_{\vec{v} \in P_d} \gcd(\vec{v}) \prod\limits_{\varepsilon_i = 0} w_{d_i} \prod\limits_{\varepsilon_j = 1}(- f_{d_j})    \end{array} \right. \]
in which $d >1$ is a divisor of $m$, and for $\mathbb{P}_d:= \{ p_i : e_{p_i}(K/\QQ) = d \}$, 

\[w_{d}:=\prod_{p \in \mathbb{P}_i}p \ \mbox{and} \ f_d := \frac{w_d - 1}{d}. \]

\end{lemma}

\begin{proof} Let $1= d_1 < d_2 < \dots < d_{\tau(m)} = m$ be the divisors of $m$. Then by reindexing, if necessary, we have

\begin{eqnarray*}
& & \prod_{i = 1}^{n} \left( p_i I - h_i \Sigma_{\langle e_{p_i} \rangle} \right) \\
 & = & \underbrace{(p_1 I - h_1 \Sigma_{\langle d_2 \rangle}) \dots (p_j I - h_j\Sigma_{\langle d_2 \rangle})}_{e_{p_i} = d_2} \dots \underbrace{(p_l I - h_l\Sigma_{\langle m \rangle}) \dots (p_{n} I - h_{n}\Sigma_{\langle m \rangle})}_{e_{p_i} = m} \\
& = &  \left(w_{d_2} I - f_{d_2} \Sigma_{\langle d_2 \rangle} \right)\ldots \left(w_m I - f_m \Sigma_{\langle m \rangle } \right).
\end{eqnarray*}

Using the fact that

\[\Sigma_{\langle d_1 \rangle} \Sigma_{\langle d_2 \rangle} = \gcd(d_1,d_2) \Sigma_{\langle \lcm[d_1,d_2] \rangle} \]
we finish by induction on the number of $f_j \neq 0$.\\
\end{proof}

\begin{corollary} Let $K$ as in Lemma \ref{coef3}. For every $d | m$, $d \geq 1$ let $A_d$ be the $m \times m $ matrix defined by
\[(A_d)_{i,j} :=  \left\{ \begin{array}{ll} 1 & \mbox{ if } \frac{m}{d} \big| (i-j) \\
& \\
0 & otherwise  \end{array} \right.\]

Then, there is a normal integral basis $\mathbb{e}:=\{\mathbb{e}_{1},...,\mathbb{e}_{m}\}$ of $\oo_K$  such that the Gram matrix of $\left< \oo_K, \Tr_{K \QQ}() \right>$ in such basis is

\[ M = \sum_{d | m} a_d A_d. \]
\end{corollary}

\begin{proof} As in Remark \ref{LaINBDeF}, for every $1 \leq i \leq n$, $F_i$ has a normal integral basis $\mathbb{e}^{i}$ of $\oo_{F_i}$. Since $F$ is the composite of the $F_i$'s and $\mathfrak{d}(F_i)$ are pairwise co-primes, then the product of such a basis form a normal integral basis $\{ b_1, b_2, \dots, b_{[F:\QQ]}\}$ of $\oo_F$.

Finally, let $t$ be a generator of  $\Gal(K/\QQ)$, $\mathbb{e}_1 := \Tr_{F/K}(b_1)$, and $\mathbb{e}_{j+1}:= t^j(\mathbb{e}_1)$. Then $\mathbb{e}:= \{\mathbb{e_1}, \dots, \mathbb{e}_m \}$ is a normal integral basis of $\oo_K$ and 

\[ \begin{array}{llc}
\Tr_{K / \QQ}(\mathbb{e}_{j} \mathbb{e}_{j}) & = \sum\limits_{d | m} a_d & \forall j  \\
& \\
\Tr_{K / \QQ}(\mathbb{e}_{i} \mathbb{e}_{j}) & =\sum\limits_{k \big|(j -i,m)} a_{\frac{mk}{(j-i, m)}} & \mbox{ if } i \neq j  .
\end{array} \]
from which the result follows. 
\end{proof}

\subsection{Even degree}

Let $K$ be a tame cyclic number field of degree $m=2^{r_0}q_1^{r_1} \dots q_{l}^{r_l}$, where $r_i \geq 1$ and $q_i$ are odd primes for every $1 \leq i \leq l$. If $p_1 , p_2, \dots, p_{n}$ are the primes ramifying in $K$  then by Lemma \ref{cycliclem} there exist number fields $F_1, F_2, \dots, F_{l}$ such that $K \subset F_1 F_2  \dots F_{l}$, $[F_i : \QQ] = e_{p_i}(K / \QQ)$, and the only prime ramifying in $F_i$ is $p_i$.\\

We define
\[h_i:= [\QQ(\eta_{p_i}) : F_i] = \frac{p_i - 1}{e_{p_i}(K / \QQ)},\]
thus, by Lemma \ref{oneprime} for every $i$ such that $h_i$ is even, the associated circulant to $F_i$ will be equivalent to
\[ s(F_i) = p_i I - h_i \Sigma_{G^i},\]
meanwhile for every $j$ such that $h_j$ is odd the associated circulant to $F_j$ should be equivalent to
\[ s(F_j) = p_j \overline{C} - h_j \Sigma_{G^i},\]
where $\displaystyle G^i := \Gal(F_i / \QQ) \simeq \ZZ/ e_{q_i}(K/\QQ)\ZZ$.\\

Therefore, if $F:= F_1 F_2 \dots F_{l}$ then the respective associated circulant $s$ to $F$ is equivalent to

\begin{eqnarray*}
  s(F) &=& s(F_1) \cdot s(F_2) \cdot ... \cdot s(F_{l}) \\
   &=& \prod_{i=1}^{l} \left(p_i Y_i - h_i \Sigma_{G^i}.\right)
\end{eqnarray*}
where,

\[ Y_i := \left\{  \begin{array}{cc} I & \mbox{ if } h_i \mbox{ is even } \\ & \\  \overline{C} & \mbox{ if } h_i \mbox{ is odd.}   \end{array} \right. \]

Finally, since $\chi': \Gal(F / \QQ) \to \Gal(K/ \QQ)$ is surjective the associated circulant to $K$ is
\begin{eqnarray*}
 \chi'(s) & = & \prod_{i=1}^{l} \chi'\left(s(F_i) \right)\\
& = & \prod_{i=1}^{l} \chi'\left(p_i Y_i - h_i \Sigma_{G^i} \right)\\
& = & \prod_{i=1}^{l} \left( p_i Y_i' - h_i \Sigma_{\langle e_{p_i} \rangle} \right).
\end{eqnarray*}

Where $\left< e_{p_i} \right>$ denotes the unique subgroup of $\ZZ / m\ZZ$ of order $e_{p_i}(K / \QQ)$,

\[ Y_i' := \left\{  \begin{array}{cc} I & \mbox{ if } h_i \mbox{ is even } \\ & \\  \sigma & \mbox{ if } h_i \mbox{ is odd}   \end{array} \right.  \]
and $\sigma$ represent the only element in $\Gal(K / \QQ)$ of order 2; For example, if $K$ is totally complex then $\sigma$ will be the complex conjugation.

\begin{theorem}{\label{even}}

Let $K, K'$ be two tame cyclic number fields with the same even degree. Then,
\[\left<\oo_K, \Tr_{K/ \QQ}  \right> \simeq \left<\oo_{K'}, \Tr_{K'/ \QQ} \right> \ \mbox{if and only if} \ \mathfrak{d}(K) = \mathfrak{d}(K').\]
\end{theorem}

\begin{proof}
As usual we only show the non trivial implication. As we have seen before the hypotheses imply that  $e_p(K / \QQ) = e_p(K' / \QQ)$ for all prime $p$. Then,
\[ h_i(K) = \frac{q_i - 1}{e_{p_i}(K / \QQ)} = \frac{q_i - 1}{e_{p_i}(K' / \QQ)} = h_i(K')\]
 for all $i$. Therefore, the respective associated circulants

\[s(K) = \prod_{i=1}^{l} \left( p_i Y_i' - h_i \Sigma_{\langle e_{p_i} \rangle} \right)  = s(K') \]
are equal. Thus, the integral quadratic modules $\left<\oo_{K}, \Tr_{K / \QQ} \right>$, $\left<\oo_{K'}, \Tr_{K' / \QQ} \right>$ are isometric.
 \end{proof}

\section{Acknowledgments}

We would like to thank the referee for the careful reading of the paper, and for their helpful comments.

\noindent
{\footnotesize Wilmar Bola\~nos, Department of Mathematics, Universidad de los
Andes,  Bogot\'a, Colombia\\ (\texttt{wr.bolanos915@uniandes.edu.co})}

\noindent
{\footnotesize Guillermo Mantilla-Soler, Department of Mathematics, Universidad Konrad Lorenz,
Bogot\'a, Colombia. Department of Mathematics and Systems Analysis, Aalto University, Espoo, Finland. ({\tt gmantelia@gmail.com})}

\end{document}